\documentclass{article}

\usepackage{amssymb,amsmath,amsthm,mathrsfs}
\usepackage{rotating}
\usepackage{verbatim}
\usepackage[a4paper,includeheadfoot,left=30mm,right=30mm,top=25mm,bottom=30mm]{geometry} 

\newtheorem{thm}{Theorem}[section]
\newtheorem{lem}[thm]{Lemma}
\newtheorem{prop}[thm]{Proposition}

\newtheorem{defn}[thm]{Definition}

\theoremstyle{definition}

\theoremstyle{remark}
\newtheorem{rmk}[thm]{Remark}

\newcommand{\res}{\operatorname*{Res}}                             
\newcommand{\Clos}{\operatorname*{clos}}                           
\renewcommand{\Re}{\operatorname*{Re}}                             
\newcommand{\dist}{\operatorname*{dist}}					                 
\renewcommand{\d}{\ensuremath{\,\mathrm{d}}}							         
\newcommand{\DeltaP}{\Delta_{\mathrm{PDE}}\hspace{0.5mm}}          
\newcommand{\Mspacer}{\hspace{0.5mm}}                              
\newcommand{\M}[3]{#1_{#2\Mspacer#3}}                              
\renewcommand{\leq}{\leqslant}                                     
\newcommand{\BE}{\begin{equation}}                                 
\newcommand{\EE}{\end{equation}}                                   
\newcommand{\BES}{\begin{equation*}}                               
\newcommand{\EES}{\end{equation*}}                                 
\newcommand{\BP}{\begin{pmatrix}}                                  
\newcommand{\EP}{\end{pmatrix}}                                    
\newcommand{\N}{\mathbb{N}}                                        
\newcommand{\R}{\mathbb{R}}                                        
\newcommand{\C}{\mathbb{C}}                                        

\newcommand{\correctionemph}[1]{\emph{#1}}                         

\newcommand{\superscript}[1]{\ensuremath{^{\textrm{#1}}}}

\newcommand{\Thns}[0]{\superscript{th}}
\newcommand{\Th}[0]{\Thns~}

\newcommand{\rdns}[0]{\superscript{rd}}
\newcommand{\rd}[0]{\rdns~}

\def\clap#1{\hbox to 0pt{\hss#1\hss}}

\numberwithin{equation}{section}

\author{David A. Smith \\ ACMAC, University of Crete, Heraklion 71003, Crete, Greece \\ email\textup{: \texttt{d.a.smith@acmac.uoc.gr}}}

\title{Well-posedness and conditioning of 3\rd and higher \\ order two-point initial-boundary value problems}

\begin{document}
\maketitle

\abstract{We discuss initial-boundary value problems of arbitrary spatial order subject to arbitrary boundary conditions. We formalise the concept of the \emph{conditioning} of such a problem and show that it represents a necessary criterion for well-posedness. The other requirement for well-posedness, the convergence of certain series, is also analysed. We illustrate these results with a full classification of 3\rd order problems having non-Robin boundary conditions.

Part of this work is devoted to correcting an oversight in the author's earlier work \emph{Well-posed two-point initial-boundary value problems with arbitrary boundary conditions} in volume~152 of Math. Proc. Camb. Philos. Soc.}

\section{Introduction} \label{sec:Intro}

It was claimed in~\cite{Smi2012a} that the decay of certain mermorphic functions within given sectors and away from their poles is sufficient for well-posedness of an initial-boundary value problem.

However, this is incorrect without an additional condition, as shown by the counterexample given in section~\ref{sec:OldIncorrect}. This was implicit in~\cite[Proposition~4$\cdot$1]{FP2001a}, and we note here that in our previous result,~\cite[Lemma~3$\cdot$1]{Smi2012a}, we omit to justify convergence of a series arising from infinitely many zeros of $\DeltaP$ lying within the closure of $D$.

In this note, we correct this oversight and adjust the statements of results in~\cite{Smi2012a} to account for this correction. We also give a full classification of 3\rd order problems with non-Robin boundary conditions. The notation and definitions are all given in~\cite{Smi2012a}.

\subsection{Conditioning of a boundary value problem}

\begin{defn} \label{defn:Wellcond}
An initial-boundary value problem is \emph{well-conditioned} if it satisfies:

\emph{$\eta_j(\rho)$ is entire and the ratio
\BE \label{eqn:P1:Intro:thm.WellPosed:Decay}
\frac{\eta_j(\rho)}{\DeltaP(\rho)}\to0 \mbox{ as }\rho\to\infty \mbox{ from within } D, \mbox{ away from the zeros of } \DeltaP.
\EE
for each $j\in\{1\ldots,2n\}$.}

Otherwise, we say that the problem is \emph{ill-conditioned}.
\end{defn}

Essentially,~\cite{Smi2012a} can be corrected by replacing references to well- or ill-posed with the notion of well- or ill-conditioned problems; we go further by providing the correct necessary and sufficient conditions for an initial-boundary value problem to be well-posed. The following is the correct statement about well-posedness of an initial-boundary value problem.

\begin{thm} \label{thm:Wellposed}
Let $(\sigma_k)_{k\in\N}$ be a sequence containing each nonzero zero of the PDE characteristic matrix $\DeltaP$ associated with an initial-boundary value problem. Then the problem is well-posed if and only if it is well-conditioned and
\BE \label{eqn:Wellposed.Bound.on.zeros}
(\sigma_k)_k\in\N \mbox{ is asymptotically within } \Clos{(E)} \mbox{ with } \dist(\sigma_k,E)=O(k^{-(n-1)}).
\EE
\end{thm}

The asymptotic bound on the positions of the zeros of $\DeltaP$ is used to ensure the convergence of the series
\BE \label{eqn:SeriesThatMustDecay}
\sum_{k\in K^+}\res_{\rho=\sigma_k}\frac{e^{i\rho x-a\rho^nt}}{\DeltaP(\rho)} \sum_{j\in J^+} \zeta_j(\rho), \qquad
\sum_{k\in K^-}\res_{\rho=\sigma_k}\frac{e^{i\rho(x-1)-a\rho^nt}}{\DeltaP(\rho)} \sum_{j\in J^-} \zeta_j(\rho),
\EE
which appear in the representation of the solution~\cite[Theorem~1$\cdot$3]{Smi2012a}.

Of the criteria for well-posedness: well-conditioning and the asymptotic bound on the zeros of $\DeltaP$, neither is stronger than the other (see Remark~\ref{rmk:Neither.Stronger}). Nevertheless, the satisfaction/failure of each criterion may be derived from the form of $\DeltaP$, so it may be possible to unite the two criteria.

\begin{rmk}
The theorem of Hille and Yosida (see, for example,~\cite{Paz1983a}) provides necessary and sufficient conditions for an initial-boundary value problem to be well-posed on the semi-infinite-time domain $(x,t)\in[0,1]\times[0,\infty)$. In the finite-time setting we consider, it is not necessary that the associated semigroup of solution operators be a contraction semigroup; infinite-time blowup of the solution is of no concern. Indeed, as the time variable is bounded, the solution operators do not even form a semigroup. Nevertheless, the weaker bound on the zeros of $\DeltaP$ is necessary to prevent the \emph{instantaneous} blowup of the solution familiar from the reverse-time heat equation. This may be compared with the theory of quasicontraction semigroups, as described in~\cite[Chapter~12]{RR2004a}.
\end{rmk}

\subsection{3\rd order non-Robin problems} \label{ssec:Intro.3nRBC}
We give a complete classification of 3\rd order initial-boundary value problems with non-Robin
boundary conditions. We show that these problems fall into six classes.
\begin{itemize}
  \item[{\bf(I)}]{The zeros of $\DeltaP$ obey~\eqref{eqn:Wellposed.Bound.on.zeros} for both $a=i$ and $a=-i$. \\ The problem is well-conditioned for neither $a=i$ nor $a=-i$.}
  \item[{\bf(II)}]{The zeros of $\DeltaP$ obey~\eqref{eqn:Wellposed.Bound.on.zeros} for $a=i$ only. \\ The problem is well-conditioned for $a=i$ only.}
  \item[{\bf(III)}]{The zeros of $\DeltaP$ obey~\eqref{eqn:Wellposed.Bound.on.zeros} for $a=-i$ only. \\ The problem is well-conditioned for $a=-i$ only.}
  \item[{\bf(IVa)}]{The zeros of $\DeltaP$ obey~\eqref{eqn:Wellposed.Bound.on.zeros} for $a=i$ only. \\ The problem is well-conditioned for both $a=i$ and $a=-i$.}
  \item[{\bf(IVb)}]{The zeros of $\DeltaP$ obey~\eqref{eqn:Wellposed.Bound.on.zeros} for both $a=i$ and $a=-i$. \\ The problem is well-conditioned for both $a=i$ and $a=-i$.}
  \item[{\bf(IVc)}]{The zeros of $\DeltaP$ obey~\eqref{eqn:Wellposed.Bound.on.zeros} for $a=-i$ only. \\ The problem is well-conditioned for both $a=i$ and $a=-i$.}
\end{itemize}
Class~{\bf(I)} corresponds to problems that are ill-posed for both direction coefficients. Classes~{\bf(II)} and~{\bf(III)} contain problems that are well-posed in one direction only and whose solutions do not admit representation by a series. The solution of a problem belonging to class~{\bf(IV)} may be represented without an integral. Only problems of class~{\bf(IVb)} are well-posed in both directions.

In Section~\ref{sec:3nRBC} we provide a method of determining to which class a given problem belongs by inspection of the boundary conditions. This expands upon the work begun in~\cite[Section~5$\cdot$1]{Smi2012a}.

\begin{rmk}
It is possible to extend this classification to non-Robin problems and to arbitrary odd order problems without great modification. However, for even order problems the situation is quite different. This is due to the power of~\cite[Theorem~1$\cdot$5]{Smi2012a} (see also Appendix~\ref{sec:Corrections}) and the failure of~\cite[Theorem~6$\cdot$3]{Smi2012a} (see~\cite[Remark~6$\cdot$4]{Smi2012a}) for even order problems.
\end{rmk}

\section{Characterisation of well-posedness} \label{sec:WellposedCorrect}

We provide an amendment of~\cite[Theorem~1$\cdot$1]{Smi2012a}. This corrected form is then refined to give the more useful characterisation of well-posedness appearing in Theorem~\ref{thm:Wellposed} of the present note.

\begin{lem} \label{lem:Wellposed1}
An initial-boundary value problem is well-posed if and only if it is well-conditioned and the series~\eqref{eqn:SeriesThatMustDecay} both converge.
\end{lem}

\begin{proof}
Given these properties, the proof of~\cite[Lemma~3$\cdot$1]{Smi2012a} can be corrected. The additional condition of the decay of these series ensures that statement below equation~(3$\cdot$8) holds.

If a problem is well-posed then its solution must be expressible in a form using these series (the original~\cite[Theorem~1$\cdot$3]{Smi2012a} has a typographical error, see the correction in Appendix~\ref{sec:Corrections}) therefore their convergence is necessary for well-posedness.
\end{proof}

\begin{proof}[Proof of Theorem~\ref{thm:Wellposed}]
Assume that the bound on the locations of the zeros of $\DeltaP$ does not hold. Then for all $R>0$ there exists some $k\in K^+$ such that
\BE
\dist(\sigma_k,\{\rho\in\C:\Re(a\rho^n)>0\})>\frac{R}{k^{n-1}}.
\EE
(Note that $K^+\cup K^-=\N$ and if there is some such $k\in K^-$ then, by~\cite[Lemma~6$\cdot$1]{Smi2012a}, there is also such a $k\in K^+$.) Hence
\BE
\dist(-a\sigma_k^n,\{\rho\in\C:\Re(\rho)<0\})>R
\EE
and $\{\Re(-a\sigma_k^n):k\in K^+\}$ is unbounded above. But then, for general initial data, the first series~\eqref{eqn:SeriesThatMustDecay} diverges for all $t>0$.

Inverting the above geometrical argument, it is immediate that if the asymptotic bound on the zeros of $\DeltaP$ holds then the real parts of the exponents $-a\sigma_k^nt$ are bounded above by $\max\{0,RT\}$, where $T$ is the final time. The exponentials $e^{i\sigma_kx}$ and $e^{i\sigma_k(x-1)}$ are bounded for $k\in K^+$ and $k\in K^-$ respectively. The series
\BE
\sum_{k\in K^+}\res_{\rho=\sigma_k}\sum_{j\in J^+} \frac{\zeta_j(\rho)}{\DeltaP(\rho)}, \qquad
\sum_{k\in K^-}\res_{\rho=\sigma_k}\sum_{j\in J^-} \frac{\zeta_j(\rho)}{\DeltaP(\rho)}
\EE
represent an expansion of the initial datum in a system of generalized eigenfunctions of the spatial differential operator. The minimality of this system guarantees their convergence for all initial data and convergence of the full series~\eqref{eqn:SeriesThatMustDecay} follows.
\end{proof}

\section{Classification of 3\rd order non-Robin problems} \label{sec:3nRBC}

We consider problems $\Pi(3,A,a,h,q_0)$ for which $A$ specifies non-Robin boundary conditions in the sense that each boundary condition contains only one order of partial derivative. Despite excluding boundary conditions of mixed order, this category is sufficiently rich in interesting examples to capture the variety of the full class while avoiding some of the notational complexities required to discuss the latter.

For such problems, the PDE characteristic determinant has one of the four forms:
\begin{subequations} \label{eqn:DeltaP.all}
\begin{alignat}{2} \label{eqn:DeltaP.I}
\mbox{\bf(I)} &\hspace{0.6em}   & \DeltaP(\rho) &= M(\rho), \\ \label{eqn:DeltaP.II}
\mbox{\bf(II)} &\hspace{0.6em}  & \DeltaP(\rho) &= M(\rho)\left[ X + \sum_{r=0}^2\omega^{Wr}e^{i\omega^r\rho} \right], \\ \label{eqn:DeltaP.III}
\mbox{\bf(III)} &\hspace{0.6em} & \DeltaP(\rho) &= M(\rho)\left[ X + \sum_{r=0}^2\omega^{Wr}e^{-i\omega^r\rho} \right], \\ \label{eqn:DeltaP.IV}
\mbox{\bf(IV)} &\hspace{0.6em}  & \DeltaP(\rho) &= M(\rho)\left[ X + \sum_{r=0}^2\omega^{Wr}\left( e^{i\omega^r\rho} + Ye^{-i\omega^r\rho} \right) \right], \hspace{8.2em}
\end{alignat}
\end{subequations}
where $M$ is a monomial with $1\leq\deg M\leq5$, $X\in\C$, $Y\in\R\setminus\{0\}$ and $W\in\{-1,0,1\}$.

\begin{prop} \label{prop:3nRBC.ZerosDeltaP}
The zeros of $\DeltaP$ occur at $\lambda_0=0$, $\omega^j\lambda_k$ and $\omega^j\mu_k$, for $j\in\{0,1,2\}$, $k\in\N$, where
\begin{itemize}
  \item[{\bf(I)}]{No such $\lambda_k$ or $\mu_k$.}
  \item[{\bf(II)}]{If $X=0$ then
    \BE
    \lambda_k = \left(2k-1-\frac{2W}{3}\right)\frac{\pi}{\sqrt{3}} + O(e^{-\sqrt{3}k\pi}).
    \EE
    If $X\neq0$ then
    \BE
    \lambda_k = \left(2k-1-\frac{2W}{3}\right)\frac{\pi}{\sqrt{3}} + O(e^{-k\pi/\sqrt{3}}).
    \EE
    No such $\mu_k$.}
  \item[{\bf(III)}]{If $X=0$ then
    \BE
    \lambda_k = -\left(2k-1-\frac{2W}{3}\right)\frac{\pi}{\sqrt{3}} + O(e^{-\sqrt{3}k\pi}).
    \EE
    If $X\neq0$ then
    \BE
    \lambda_k = -\left(2k-1-\frac{2W}{3}\right)\frac{\pi}{\sqrt{3}} + O(e^{-k\pi/\sqrt{3}}).
    \EE
    No such $\mu_k$.}
  \item[{\bf(IV)}]{If $Y>0$ then
    \begin{subequations}
    \begin{align}
    \lambda_k &= \left(2k-1+\frac{2W}{3}\right)\pi + i\log|Y| + O(e^{-\sqrt{3}k\pi}), \\
    \mu_k &= -\left(2k-1+\frac{2W}{3}\right)\pi + i\log|Y| + O(e^{-\sqrt{3}k\pi}).
    \end{align}
    \end{subequations}
    If $Y<0$ then
    \begin{subequations}
    \begin{align}
    \lambda_k &= \left(2k+\frac{2W'}{3}\right)\pi + i\log|Y| + O(e^{-\sqrt{3}k\pi}), \\
    \mu_k &= -\left(2k+\frac{2W'}{3}\right)\pi + i\log|Y| + O(e^{-\sqrt{3}k\pi}),
    \end{align}
    \end{subequations}
    where $W'\in\{-2,-1,0\}$ such that $W'=W$ $(\mathrm{mod}$ $3)$.}
\end{itemize}
\end{prop}

\begin{proof}
{\bf(I)} is immediate.

By the symmetry of the expressions~\eqref{eqn:DeltaP.all}, if $\DeltaP(\rho)=0$ then $\DeltaP(\omega\rho)=0$.

By the results of~\cite[Section~6]{Smi2012a}, the only solutions of~\eqref{eqn:DeltaP.II} (up to rotational symmetry of order 3) lie asymptotically on a ray parallel to the positive imaginary axis. Solving~\eqref{eqn:DeltaP.II} for $\lambda_k=xi +y+z$ with $\R^+\ni x=O(k)$, $\R\ni y=$ constant, $\C\ni z=o(1)$, we obtain
\BES
\omega^W e^{i\omega\rho} + \omega^{2W} e^{i\omega^2\rho} = \begin{cases}O(1) & X\neq0, \\ O(e^{-x}) & X=0, \end{cases}
\EES
hence
\BES
1 + \omega^{W} e^{\sqrt{3}(xi+y+z)} = \begin{cases}O(e^{-x/2}) & X\neq0, \\ O(e^{-3x/2}) & X=0. \end{cases}
\EES
But then $y=0$,
\BES
\sqrt{3}x + \frac{2W\pi}{3} = (2k-1)\pi \quad\mbox{and}\quad z = \begin{cases}O(e^{-x/2}) & X\neq0, \\ O(e^{-3x/2}) & X=0 \end{cases}
\EES
and result~{\bf(II)} follows.

Result~{\bf(III)} follows from~{\bf(II)} using the map $\rho\mapsto-\rho$.

A similar argument justifies~{\bf(IV)} but we must consider zeros lying on two rays parallel to the positive and negative real axes.
\end{proof}

\begin{prop} \label{prop:3nRBC.ClassificationDeltaP}
The forms of $\DeltaP$ \textup{\bf(I)}--\textup{\bf(IV)} above correspond to the classes \textup{\bf(I)}--\textup{\bf(IV)} defined in section~\ref{ssec:Intro.3nRBC} with \textup{\bf(IVa)} if $|Y|<1$, \textup{\bf(IVb)} if $|Y|=1$ and \textup{\bf(IVc)} if $|Y|>1$.
\end{prop}

\begin{proof}
The asymptotic location of the zeros of $\DeltaP$ is given in Proposition~\ref{prop:3nRBC.ZerosDeltaP}.

We consider the form~{\bf(III)}. If $4\pi/3<\arg\rho<5\pi/3$ then, for some $j\in\{1,\ldots,2n\}$ and some constants $c,c'\in\C$,
\BE
\frac{\eta_j(\rho)}{\DeltaP(\rho)} = \int_0^1{\frac{ce^{-i(\omega+\omega^2x)\rho}+c'e^{-i(\omega^2+\omega x)\rho} + o\left(e^{|\rho|/2}\right)}{\omega^W e^{-i\omega\rho} + \omega^{2W} e^{-i\omega^2\rho} + O(1)}}q_T(x)\d x.
\EE
Hence this ratio blows up if
\BE \label{eqn:Blowup.Cond}
D \cap \{ \rho\in\C: 4\pi/3<\arg\rho<5\pi/3 \} \neq \emptyset,
\EE
that is, if $a\neq-i$.

However, if $\pi/3<\arg\rho<2\pi/3$ then
\BE
\frac{\eta_j(\rho)}{\DeltaP(\rho)} = \int_0^1{\frac{ce^{-i(1+\omega x)\rho}+c'e^{-i(1+\omega^2x)\rho} + o\left(e^{|\rho|/2}\right)}{\omega^W e^{-i\omega\rho} + \omega^{2W} e^{-i\omega^2\rho} + O(1)}}q_T(x)\d x,
\EE
and this ratio decays provided
\BE
D^+ \subset \{ \rho\in\C: \pi/3<\arg\rho<2\pi/3 \},
\EE
that is, if $a=-i$. The same argument can be applied to the sectors $\pi<\arg\rho<4\pi/3$ and $5\pi/3<\arg\rho<2\pi$.

We have shown that problems of class~{\bf(III)} have the claimed conditioning. The conditioning of problems belonging to class~{\bf(II)} follows immediately using the map $\rho\mapsto-\rho$. For problems in class~{\bf(IV)}, the exponentials in the numerator of $\eta_j/\DeltaP$ can never blow up faster than those in the denominator, within any sector.
\end{proof}

\begin{rmk} \label{rmk:Neither.Stronger}
For odd-order problems in both the present work and the work~\cite{Smi2012a}, we permit only $a=\pm i$. However the requirement~\eqref{eqn:Blowup.Cond} derived in the above proof explicitly forces $\arg a=-\pi/2$ for a well-conditioned problem belonging to class~{\bf(III)}. However, for the zeros of $\DeltaP$ to obey~\eqref{eqn:Wellposed.Bound.on.zeros} in a class~{\bf(III)} problem it is sufficient that $-\pi\leq\arg a\leq0$. In problems belonging to class~{\bf(III)} (and, similarly, class~{\bf(II)}) well-conditioning is a strictly stronger requirement than the asymptotic bound on the location of the zeros.

In contrast, any problem that is well-conditioned for some particular coupling constant $a$ and also for $-a$ must be well-conditioned for \emph{all} nonzero complex coupling constants. Problems belonging to class~{\bf(IV)} have this property. However, because the zeros of $\DeltaP$ lie on 6 rays, it is necessary that $a=\pm i$ for well-posedness. In problems of class~{\bf(IV)}, the asymptotic bound~\eqref{eqn:Wellposed.Bound.on.zeros} on the location of the zeros is strictly stronger than well-conditioning.
\end{rmk}

\subsection{Pseudo-periodic problems}
We expand on the pseudo-periodic~\cite[Example~5$\cdot$2]{Smi2012a}, showing how the values of the coupling constants affect to which class this problem belongs.

Let
\BE
A = \BP 1&\beta_1&0&0&0&0\\0&0&1&\beta_2&0&0\\0&0&0&0&1&\beta_2 \EP,
\EE
for some nonzero real coupling constants $\beta_1$, $\beta_2$ and $\beta_3$. Then
\BE
\DeltaP(\rho) = a\rho^3(\omega^2-\omega) \left[ \beta'' + \sum_{r=0}^2\left( \beta' e^{i\omega^r\rho} + \beta e^{-i\omega^r\rho} \right) \right],
\EE
where
\begin{subequations}
\begin{align}
\beta   &= \beta_1\beta_2+\beta_2\beta_3+\beta_3\beta_1, \\
\beta'  &= \beta_1+\beta_2+\beta_3, \\
\beta'' &= 3(\beta_1\beta_2\beta_3+1).
\end{align}
\end{subequations}
As all the coupling constants are real, no more than one of the quantities $\beta$, $\beta'$, $\beta''$ may be $0$ for any particular problem. Hence no pseudo-periodic problems belong to class~(I).

If $\beta=0$ then the problem belongs to class~{\bf(II)}, with
\BE
X = \beta''/\beta', \qquad W=0.
\EE
If $\beta'=0$ then the problem belongs to class~{\bf(III)}, with
\BE
X = \beta''/\beta, \qquad W=0.
\EE
Finally, if $\beta,\beta'\neq0$ then the problem belongs to class~{\bf(IV)}, with
\BE
X = \beta''/\beta', \qquad Y=\beta/\beta', \qquad W=0.
\EE
These results are summarised in Table~\ref{tab:3nRBC.2.3}

\subsection{Other problems}

\begin{table}
\centering
\caption{3\rd order non-Robin problems: 0 or 1 couplings}
\label{tab:3nRBC.0.1}
\bigskip

{\small
\begin{tabular}{clcccc}
Couplings & Description        & Class      &$W$& $X$               & $Y$ \\ \hline

0 & all BC at one side         & {\bf(I)}   &   &                   & \\ \hline

  & 2 BC at left,              &            &   &                   & \\
0 & 1 at right,                & {\bf(II)}  & 0 & 0                 & \\
  & all BC of different orders &            &   &                   & \\ \hline

  & 2 BC at left, 1 at right   &            &   &                   & \\
  & 2 of same order,           &            &   &                   & \\
0 & order(left \& right BC)    & {\bf(II)}  & 1 & 0                 & \\
  & $=$order(left BC)$-1$      &            &   &                   & \\
  & (mod 3)                    &            &   &                   & \\ \hline

  & 2 BC at left, 1 at right   &            &   &                   & \\
  & 2 of same order,           &            &   &                   & \\
0 & order(left \& right BC)    & {\bf(II)}  & -1& 0                 & \\
  & $=$order(left BC)$+1$      &            &   &                   & \\
  & (mod 3)                    &            &   &                   & \\ \hline

  & 1 BC at left,              &            &   &                   & \\
0 & 2 at right,                & {\bf(III)} & 0 & 0                 & \\
  & all BC of different orders &            &   &                   & \\ \hline

  & 1 BC at left, 2 at right   &            &   &                   & \\
  & 2 of same order,           &            &   &                   & \\
0 & order(left \& right BC)    & {\bf(III)} & 1 & 0                 & \\
  & $=$order(left BC)$-1$      &            &   &                   & \\
  & (mod 3)                    &            &   &                   & \\ \hline

  & 1 BC at left, 2 at right   &            &   &                   & \\
  & 2 of same order,           &            &   &                   & \\
0 & order(left \& right BC)    & {\bf(III)} & -1& 0                 & \\
  & $=$order(left BC)$+1$      &            &   &                   & \\
  & (mod 3)                    &            &   &                   & \\ \hline

1 & 2 BC at left               & {\bf(II)}  & 0 & $3/\beta_1$       & \\ \hline

1 & 2 BC at right              & {\bf(III)} & 0 & $3 \beta_1$       & \\ \hline

  & 1 BC at left,              &            &   &                   & \\
1 & 1 at right,                & {\bf(IV)}  & 0 & 0                 & $\beta_1$ \\
  & all BC of different orders &            &   &                   & \\ \hline
  
  & 1 BC at left, 1 at right   &            &   &                   & \\
  & of same order,             &            &   &                   & \\
1 & order(uncoupled BC)        & {\bf(IV)}  & 1 & 0                 & $-\beta_1$ \\
  & $=$order(coupled BC)$-1$   &            &   &                   & \\
  & (mod 3)                    &            &   &                   & \\ \hline

  & 1 BC at left, 1 at right   &            &   &                   & \\
  & of same order,             &            &   &                   & \\
1 & order(uncoupled BC)        & {\bf(IV)}  & -1& 0                 & $-\beta_1$ \\
  & $=$order(coupled BC)$+1$   &            &   &                   & \\
  & (mod 3)                    &            &   &                   & \\ 

\end{tabular}
}
\end{table}

\begin{table}
\centering
\caption{3\rd order non-Robin problems: 2 or 3 couplings}
\label{tab:3nRBC.2.3}
\bigskip

{\small
\begin{tabular}{clcccc}
Couplings & Description        & Class      &$W$& $X$               & $Y$ \\ \hline

  & 1 BC at right              &            &   &                   & \\
2 & and                        & {\bf(II)}  & 0 & $3\beta_1\beta_2$ & \\
  & $\beta_1+\beta_2=0$        &            &   &                   & \\ \hline

  & 1 BC at left               &            &   &                   & \\
2 & and                        & {\bf(III)} & 0 & $3/\beta_1\beta_2$& \\
  & $\beta_1+\beta_2=0$        &            &   &                   & \\ \hline

  & 1 BC at right              &            &   &                   & \\
2 & and                        & {\bf(IV)}  & 0 & $3\beta_1\beta_2$ & $\beta_1+\beta_2$ \\
  & $\beta_1+\beta_2\neq0$     &            &   &                   & \\ \hline

  & 1 BC at left               &            &   &                   & \\
2 & and                        & {\bf(IV)}  & 0 &$3/(\beta_1+\beta_2)$& $\beta_1\beta_2/(\beta_1+\beta_2)$ \\
  & $\beta_1+\beta_2\neq0$     &            &   &                   & \\ \hline

3 & $\beta=0$                  & {\bf(II)}  & 0 & $\beta''/\beta'$  & \\ \hline

3 & $\beta'=0$                 & {\bf(III)} & 0 & $\beta''/\beta$   & \\ \hline

3 & $\beta,\beta'\neq0$        & {\bf(IV)}  & 0 & $\beta''/\beta'$  & $\beta/\beta'$ \\ 

\end{tabular}
}
\end{table}

A 3\rd order problem with non-Robin boundary conditions need not be pseudo-periodic; there may be two, one or no coupled boundary conditions. Fully uncoupled problems were discussed in~\cite{Pel2004a} but we include them here for comparison. For brevity, we present the results in Tables~\ref{tab:3nRBC.0.1}--\ref{tab:3nRBC.2.3}, omitting the full derivations. We denote the $k$ coupling constants $\beta_1,\ldots,\beta_k$; the ordering is irrelevant.

Recall from Proposition~\ref{prop:3nRBC.ClassificationDeltaP} that $|Y|$ discriminates between classes~{\bf(IVa)}--{\bf(IVc)}.

\section{A counterexample} \label{sec:OldIncorrect}

If the arguments of~\cite[Section~3]{Smi2012a} are sound then it must be that, for any $\beta\in(0,1)$, the following initial-boundary value problems
\begin{align*}
(\partial_t + i(-i\partial_x)^3) q(x,t) &= 0, & & & (\partial_t - i(-i\partial_x)^3) q(x,t) &= 0, \\
q(x,0) &= q_0(x), & & & q(x,0) &= q_0(x), \\
q(0,t) = q(1,t) &= 0, & & & q(0,t) = q(1,t) &= 0, \\
\partial_xq(0,t) &= \beta \partial_xq(1,t), & & & \partial_xq(0,t) &= \beta \partial_xq(1,t)
\end{align*}
each have unique $C^\infty[0,1]\times[0,T]$ solutions that can be represented as series expansions
\begin{align*}
q(x,t) &= \sum_{k\in K^+} e^{i\sigma_kx - i\sigma_k^3t} f_k(q_0) & \quad q(x,t) &= \sum_{k\in K^+} e^{i\sigma_kx + i\sigma_k^3t} f_k(q_0) \\
&\hspace{5ex}+ \sum_{k\in K^-} e^{i\sigma_k(x-1) - i\sigma_k^3t} f_k(q_0),\quad & &\hspace{5ex}+ \sum_{k\in K^-} e^{i\sigma_k(x-1) - i\sigma_k^3t} f_k(q_0),
\end{align*}
where $\sigma_k\in\C^\pm$ are the zeros of the PDE characteristic determinant indexed by $k\in K^\pm$.

These series converge or diverge depending upon the asymptotic location of the $\sigma_k$. If $t=0$ then the series converge as $\Re(i\sigma_k x)<0$ for $k\in K^+$ and $\Re(i\sigma_k (x-1))<0$ for $k\in K^-$. It remains to consider $\Re(\pm i\sigma_k^3t)$.

By~\cite[Lemma~6.1]{Smi2012a} $\DeltaP(\omega\rho)=0$ if and only if $\DeltaP(\rho)=0$ for $\omega^3=1$ so we can reparametrize these series so that we only sum over one of the three cube roots of $\sigma_k^3$. Evaluating $\DeltaP$ yields two sequences of zeros $(\lambda_k)_{k\in\N}$ and $(\mu_k)_{k\in\N}$ such that
\BES
\{0\}\cup\{\omega^j\lambda_k,\omega^j\mu_k:j\in\{0,1,2\},k\in\N\}
\EES
contains each zero of $\DeltaP$ precisely once and
\begin{align*}
\lambda_k &= (2k-1/3)\pi + i\log\beta + O(\exp(-\sqrt{3}k\pi/2)), \\
\mu_k &= -(2k-1/3)\pi + i\log\beta + O(\exp(-\sqrt{3}k\pi/2)).
\end{align*}
Hence
\begin{align*}
\lambda_k^3 &= \left(\left[(2k-1/3)\pi\right]^3 - (2k-1/3)\pi\left[\log\beta\right]^2\right) \\
&\hspace{15ex} + i\left(\left[(2k-1/3)\pi\right]^2\log\beta-\left[\log\beta\right]^3\right) + O(k^3\exp(-\sqrt{3}k\pi/2)), \\
\mu_k^3 &= -\left(\left[(2k-1/3)\pi\right]^3 - (2k-1/3)\pi\left[\log\beta\right]^2\right) \\
&\hspace{15ex} + i\left(\left[(2k-1/3)\pi\right]^2\log\beta-\left[\log\beta\right]^3\right) + O(k^3\exp(-\sqrt{3}k\pi/2)).
\end{align*}
But then
\begin{align*}
\Re(-i\lambda_k^3) &= \left[\left(2k-1/3\right)\pi\right]^2\log\beta + O(1)\quad & \Re(i\lambda_k^3) &= -\left[\left(2k-1/3\right)\pi\right]^2\log\beta + O(1) \\
&<0, & &>0, \\
\Re(-i\mu_k^3) &= \left[\left(2k-1/3\right)\pi\right]^2\log\beta + O(1) & \Re(i\mu_k^3) &= -\left[\left(2k-1/3\right)\pi\right]^2\log\beta + O(1) \\
&<0, & &>0,
\end{align*}
so the series representation on the right hand side cannot converge for positive time. This means that the problem is ill-posed.

\appendix

\section{Errata Corrige} \label{sec:Corrections}

Owing to the correction of~\cite[Theorem~1$\cdot$1]{Smi2012a} adjustments to the statements of a number of other results in that work are necessary. We present the important corrections here. An unrelated typographical error is also noted.

The conditions of Theorem~1$\cdot$4 are sufficient but they can be weakened. Indeed, the reverse-time problem $\Pi'$ need not be well-posed, only \correctionemph{well-conditioned}.

Theorem~1$\cdot$5 should state that \correctionemph{well-conditioning} is equivalent for the two problems instead of well-posedness.

The conditions for a 4\Th order pseudoperiodic problem to be ill-posed are sufficient but \correctionemph{not necessary}. Defining
\BE
\beta = \frac{\M{\beta}{1}{3} + \M{\beta}{2}{2} + \M{\beta}{3}{1} + \M{\beta}{4}{0} + \M{\beta}{1}{3}\M{\beta}{2}{2}\M{\beta}{3}{1}\M{\beta}{4}{0}\left( \frac{1}{\M{\beta}{1}{3}} + \frac{1}{\M{\beta}{2}{2}} + \frac{1}{\M{\beta}{3}{1}} + \frac{1}{\M{\beta}{4}{0}} \right)}{\M{\beta}{1}{3}\M{\beta}{2}{2} + \M{\beta}{2}{2}\M{\beta}{3}{1} + \M{\beta}{3}{1}\M{\beta}{4}{0} + \M{\beta}{4}{0}\M{\beta}{1}{3} + 2(\M{\beta}{1}{3}\M{\beta}{3}{1} + \M{\beta}{2}{2}\M{\beta}{4}{0})},
\EE
the problem is ill-posed if and only if $|\beta|>1/2$ or the denominator is $0$.

The coefficients in Theorems~1$\cdot$3--4 are incorrect. Indeed, equation~(1$\cdot$8) should read
\begin{multline} \label{eqn:P1:Intro:thm.Reps.Int:q}
q(x,t) = i\sum_{k\in K^+}\res_{\rho=\sigma_k}\frac{e^{i\rho x-a\rho^nt}}{\DeltaP(\rho)} \sum_{j\in J^+} \zeta_j(\rho) + \frac{1}{2\pi}\int_{\partial\widetilde{E}^+}e^{i\rho x-a\rho^nt}\sum_{j\in J^+}\frac{\zeta_j(\rho)}{\DeltaP(\rho)}\d\rho \\
+ i\sum_{k\in K^-}\res_{\rho=\sigma_k}\frac{e^{i\rho(x-1)-a\rho^nt}}{\DeltaP(\rho)} \sum_{j\in J^-} \zeta_j(\rho) + \frac{1}{2\pi}\int_{\partial\widetilde{E}^-}e^{i\rho(x-1)-a\rho^nt}\sum_{j\in J^-}\frac{\zeta_j(\rho)}{\DeltaP(\rho)}\d\rho \\
- \frac{1}{2\pi}\left\{\sum_{k\in K^\mathbb{R}} \int_{\Gamma_k} + \int_{\mathbb{R}} \right\} e^{i\rho x-a\rho^nt} \left( \frac{1}{\DeltaP(\rho)}-1 \right)H(\rho) \d\rho.
\end{multline}
Equation~(1$\cdot$9) must be adjusted similarly.

\bigskip

The author is sincerely grateful to Beatrice Pelloni and Dmitry Pelinovsky for their helpful comments and criticism. This work was funded by the FP7 programme of the European Commission.

\bibliographystyle{amsplain}
\bibliography{dbrefs}
\end{document}